\documentclass[12pt]{article}
\usepackage{amsmath,amssymb,amsfonts,amsthm}
\newenvironment{myabstract}{\par\noindent
{\bf Abstract . } \small }
{\par\vskip8pt minus3pt\rm}
\newcounter{item}[section]
\newcounter{kirshr}
\newcounter{kirsha}
\newcounter{kirshb}
\newenvironment{enumarab}{\setcounter{kirshb}{1}
\begin{list}{(\arabic{kirshb})}{\usecounter{kirshb}} }{\end{list}}
\newtheorem{theorem}{Theorem}[section]

\newtheorem{lemma}[theorem]{Lemma}
\newtheorem{corollary}[theorem]{Corollary}
\newenvironment{demo}[1]{\noindent{\bf #1.}\upshape\mdseries}
{\nopagebreak{\hfill\rule{2mm}{2mm}\nopagebreak}\par\normalfont}
\theoremstyle{definition}

\newtheorem{definition}[theorem]{Definition}

\def\C{{\mathfrak{C}}}

\def\Nr{{\mathfrak{Nr}}}
\def\Fr{{\mathfrak{Fr}}}
\def\Sg{{\mathfrak{Sg}}}

\def\A{{\mathfrak{A}}}
\def\B{{\mathfrak{B}}}
\def\C{{\mathfrak{C}}}

\def\CA{{\bf CA}}

\def\(R)RA{{\bf (R)RA}}

 \def\CA{{\sf CA}}
\def\B{{\sf B}}
\def\G{{\sf G}}

 \def\Cm{{\mathfrak{Cm}}}
\def\Nr{{\mathfrak{Nr}}}

\def\Nr{{\mathfrak{Nr}}}

\def\A{{\mathfrak{A}}}
\def\B{{\mathfrak{B}}}
\def\C{{\mathfrak{C}}}

\def\A{{\mathfrak{A}}}
\def\B{{\mathfrak{B}}}
\def\C{{\mathfrak{C}}}

\def\Cm{{\mathfrak{Cm}}}

\def\At{{\mathfrak{At}}}

\def\CA{{\bf CA}}

\def\G{{\bf G}}

\def\Lm{{\sf Lm}}

\def\CPEA{{\sf CPEA}}
\def\CPES{{\sf CPES}}
\def\CPE{{\sf CPE}}
\def\Gwp{{\sf Gwp}}
\def\Gp{{\sf Gp}}
\def\PEA{{\sf PEA}}

\def\A{{\cal{A}}}
\def\B{{\mathfrak{B}}}
\def\C{{\mathfrak{C}}}

\def\G{{\mathfrak{G}}}
\def\H{{\mathfrak{H}}}

\def\Nr{{\mathfrak{Nr}}}
\def\F{{\mathfrak{F}}}

\def\CA{{\bf CA}}

\def\Cm{{\mathfrak Cm}}
\def\Sg{{\mathfrak Sg}}
\def\A{{\mathfrak A}}
\def\F{{\mathfrak F}}
\def\Str{{\mathfrak Str}}
\def\PA{{\sf PA}}

\title{Cylindric-Polyadic algebras have the super amalgamation property}
\author{Tarek Sayed Ahmed}

\begin{document}
\maketitle

\begin{abstract} We show that the  cylindric- polyadic algebras, introduced by Ferenczi, have the superamalgamation property
\end{abstract} 

\section{Introduction}

So called relativization started as a technique for generalizing representations of cylindric algebras, while also, in some cases, 
`defusing' undesirable properties, like undecidability or lack of definability (like Beth definability).
These ideas have counterparts in logic, and they have been influential in several ways. 
Relativization in cylindric-like algebras lends itself to a modal perspective where transitions are viewed as 
objects in their own right, in addition to states, while algebraic terms 
now correspond to modal formulas defining the essential properties 
of transitions. The modal perspective also suggests that first order predicate logic reflects only part of the expressive resources of abstract state 
models.

Indeed, why insist on standard models? This is a voluntary commitment to only one mathematical implementation, 
whose undesirable complexities can pollute the laws of logics needed to describe the core phenomena.
Set theoretic cartesian squares modelling as the intended  vehicle may not be an orthogonal concern, 
it can be detremental, repeating hereditory sins of old paradigms. 

Indeed in \cite{HMT2} square units got all the attention and relativization  was treated as a side issue.
Extending original classes of models for logics to manipulate their properties is common. 
This is no mere tactical opportunism, general models just do the right thing.

The famous move from standard models to generalized models is 
Henkin's turning round  second  order logic into an axiomatizable two sorted first
order logic. Such moves are most attractive 
when they get an independent motivation. 

The idea is that we want to find a semantics that gives just the bare bones of action, 
while additional effects of square set theoretic modelling are separated out as negotiable decisions of formulation 
that threatens completeness, decidability,  and interpolation.

And indeed by using relativized representations Ferenzci, proved that if we weaken commutativity of cylindrifiers
and allow  relativized representations, then we get a finitely axiomatizable variety of representable 
quasi-polyadic equality algebras (analogous to the Andr\'eka-Resek Thompson $\CA$ version); 
even more this can be done without the merry go round identities.
This is in sharp view with  complexity results proved recently by the author for quasi poyadic equality algebras, of non finite axiomatizability
over their diagonal free reduct and transposition free reducts.
Ferenzci's results can be seen as establishing a hitherto fruitful contact between neat embedding theorems and relativized representations, with enriching 
repercussions for both notions.

One can find well motivated appropriate notions of semantics by first locating them while giving up classical semantical prejudices. 
It is hard to give a precise mathematical underpinning to such intuitions. What really counts at the end of the day
is a completeness theorem stating a natural fit between chosen intuitive concrete-enough, but perhaps not {\it excessively} concrete, 
semantics and well behaved axiomatizations. 
The move of altering semantics has radical phiosophical repercussions, taking us away from the conventional 
Tarskian semantics captured by Fregean-Godel-like axiomatization; the latter completeness proof is effective 
but highly undecidable; and this property is inherited by finite variable fragments 
of first order logic as long as we insist on Tarskian semantics.

Now we use two techniques to get positive results, concerning the superamalgamation property for what Ferenzci calls 
cylindric-polyadic algebras. The term is justified by the fact that their signature
ahs only finite cylindrifiers (like cylindric algebras) but they also have {\it all} substitutions like polyadic algebras.
The first is yet again a Henkin construction (carefully implemented because we have changed the semantics, 
so that desired Henkin ultrafilters, on which our relativized models, will be based are more involved), the other is inspired by the well-developed 
duality theory in modal logic between Kripke frames and complex algebras.
This last technique was first implemented by N\'emeti in the context of relativized cylindric set algebras 
which are complex algebras of 
weak atom structures.


We need some preliminaries. The part to follow, which has to do only with representation, 
is due to Ferenczi, which we include with his permission. 
We thank him for sending us the tex file of his manuscript, which made the writing 
much easier. Using one modified version of a Henkin construction, which Ferenczi calls perfect, 
proves a very elegant completeness
theorem for infinitary extensions of first order logic with equality. 
In fact, such a result can be regarded as one possible solution to the fiitizability problem for first order logic with
equality! 

To get the stronger result of interpolation (indeed the former result can be easily destilled from the proof of the latter), 
we need {\it two} perfect ultrafilters 
that agree on the common subalgebra, or the common language of the two formulas to be interpolated.

The Daigneault--Monk--Keisler, neat embedding theorem, says that if $\A\in \PA_{\alpha}$
$\A\in S\Nr_{\alpha }\B$ 
for some $\B\in \PA_{\alpha +\varepsilon }$\textit{,
where }$\alpha $\textit{\ is a fixed infinite ordinal and }$\varepsilon >1,$, then $\A$ is representable.
(see \cite{DM} and \cite{HMT2} Thm. 5.4.17). 
The neat embedding part does work when we add diagonal elements, every algebra neatly embeds into arbitrary higher dimensions;
however, this {\it does not} enforce representability like the diagonal free case. 
Therefore, it is quite an achievement to obtain a representation theorem for polyadic-like algebras {\it with diagonal elements}, 
even if the representation is relativized. 

On the face of it, it seems that the process of relativization in this context, 
is not a free choice, it is necessary, to get over the hurdle of incompatibility of the presence of 
infinitary substitutions and diagonal elements together. This phenomena manifests itself prominently, like for example in the case of Sain's 
algebras prohibiting a solution to the finitizability problem for (algebraisable extensions of) first order logic {\it with} equality.

This unfavourable contact  is explained with care in \cite{Sagi}, showing that such a precarious  
combination blows up ultraproducts, taking us out of the 
representant class, if we stick to Tarskian semantics.  

To cut a long story short, Tarskian square semantics breaks down here, yet again, confirming our earlier 
stipulation, that square semantics was not the best choice in the world,  
and we also need to weaken the Rosser condition of commutativity of cylindrifiers,
so we are actually varying both syntax and semantics, in the ultimate aim of obtaining a perfect match 
represented in a strong completeness theorem

This is indeed, analogous to celebrated the breath-taking 
Andr\'eka-Resek-Thompson result, which opened many windows 
establishing a fruitful dichotomy in algebraic logic,
between set algebras that  are so resilieint to nice axiomtizations, 
to Stone like neat representation theorems for algebraisations of a whole landscape of multi dimensional modal logics of quantifiers, with the
ill-behaved  multi dimensional cylindric modal logic, just appearing as the top of an ice-berg, with a lot of hidden treasures below the surface.

\bigskip

This form of the neat embedding theorem for $\PA$ is due to Daigneault and
Monk. Keisler published the proof theoretical
variant of the theorem in the same issue.  Here we are going to
refer to the proof  of Theorem 4.3 in \cite{DM} and its variant for
polyadic \textit{equality} algebras (\cite{HMT2} II. 5.4.17.)

Several abstract classes $\CPEA_{\alpha}, \CPES_{\alpha}$, $_m\CPEA_{\alpha}$ ($m<\alpha)$, ${\sf Lm}_{\alpha}$ are
defined in \cite{Fer} p. 156, which are polyadic equality algebras, without full fledged commutativity of 
cylindrifiers. Te class ${\sf Lm}_{\alpha}$ are those algebras in $_m\CPEA_{\alpha}$ such that $|\Delta b|\leq \alpha$ for all $b$.  
\begin{definition}
\begin{enumarab}
\item Assume $\alpha$ is an infinite ordinal and $m<\alpha$ is infinite. Given a set $U$ and a fixes sequence $p\in {}^{\alpha}U$, the set
$$_m^{\alpha}U^{(p)}=\{x\in {}^{\alpha}U: \text { $x$ and $p$ are different at most in $m$ many spaces }\}$$
\item A transformaton $\tau$ on $\alpha$ is said to be an $m$ transformatin if $\tau i=i$ excepty for $m$-many $i\alpha$. 
The set of al such transformations
is denoted by $_m{\sf T}_{\alpha}$
\end{enumarab}
\end{definition}

Concrete classes of relativized set algebras $_m\Gwp_{\alpha}$ and $\Gp_{\alpha}$ are aso defined on the next page.
The units of such algebras are unions of squares or weak spaces, but disjointness of the bases  is omitted, and the top elements satsify certain closure conditions depending on the 
substitutions used.
So here we are encountered with a different geometry too. Such new axioms or principles 
are geometric conditions on a par with standard geometric axioms about,
points, lines, sqaures and rectangles in  multi dimensional geometric spaces.

With so many defined classes, and so many representation theorems proved by Ferenczi, 
we make the choice of sticking to one of them. The other cases can be approached in exactly the same manner
undergoing the obvious modifications, with the aid of Ferenczi's work.

The axiom refered to as (CP9)$^*$ is also defined in \cite{Fer}, which roughly states 
that cylindrifiers do not commute even with substitutions, when it is consistent that they can as in the case of 
square representations of $\PEA$.

\section{First proof}
Ferenzci shows that the Diagneault Monk  theorem  holds 
if the class $\PA_{\alpha }$ is replaced by $%
_{m}$\CPE$_{\alpha }$ and $\PEA_{\alpha +\varepsilon }$ is replaced by such a 
\textit{class} $_{m}$\CPE$_{\alpha +\varepsilon }^{-}$ such that the $_{m}$\CPE$%
_{\alpha +\varepsilon }$ axioms hold, except for the axiom (CP9)$%
^{\ast }$ which merely holds for every $i,j\in \alpha ,\sigma \in $ $_{m}$T$%
_{\alpha }$ and, in additional, the following two instances of (CP9)$^{\ast
} $ are satisfied:

\begin{equation}
c_{i}x=c_{m}s_{[i\text{ }/\text{ }m]}x\text{ if }i\in \alpha ,m\notin \alpha
,x\in A  \label{felcserélhetoA}
\end{equation}

\begin{equation}
c_{m}s_{\tau }z=s_{\tau }c_{m}z\text{ if }\tau m=m,\text{ }m\notin \alpha ,%
\text{ }\tau \in \text{ }_{m}\text{T}_{\beta },\text{ }z\in B
\label{felcserélhetoB}
\end{equation}
\newline
By $r$-representability, we mean {\it relativized} representability.
\bigskip

The following theorem implies that a more sophisticated kind of neat embeddability of an
algebra in $_{m}$\CPE$_{\alpha }\cap {\sf Lm}_{\alpha}$ is equivalent to $r$%
-representability. 

\bigskip

\begin{theorem}(Ferenzci) Assume that $\mathcal{A}\in $ $_{m}$\CPE$_{\alpha
}\cap  \Lm_{\alpha },$ \textit{where }$m$\textit{\ is infinite, }$m<\alpha $%
\textit{.} Then $\mathcal{A}\in $ S$\Nr_{\alpha }\mathcal{B}$\textit{\ for
some} $\mathcal{B}$ $\in $ $_{m}$\CPE$_{\alpha +\varepsilon }^{-},$ \textit{%
where }$\varepsilon $\textit{\ is infinite, if and only if\ }$\mathcal{A}\in 
{\bf I}_{m}$\Gwp$_{\alpha }.$

\bigskip\end{theorem}
Let us consider the hard direction, namely,  that $\mathcal{A}\in $ S$\Nr_{\alpha }\mathcal{B%
}$ implies $\mathcal{A}\in {\bf I}_{m}$\Gwp$_{\alpha }$. 

\bigskip

Fix an algebra $\mathcal{B}
$. Let us denote by $adm$ the class of $m$%
-transformations $\tau $ $\in $ $^{\alpha }\beta $, i.e., $\tau \in $ $_{m}$T%
$_{\alpha }$ $\cap $ $^{\alpha }\beta ,$ where $\alpha +\varepsilon $ is
denoted by $\beta .$ We introduce a concept needed in the proof: \textit{A
Boolean ultrafilter ${F}$ $in$ $\mathcal{B}$ is a }perfect\textit{\
ultrafilter if for any element of the form $s_{\tau }c_{j}x$ included in $F$%
, where $j\in \alpha $, $x\in A$ and $\tau \in adm$, there exists an $m,$ $%
m\notin \alpha ,$ $\tau m=m$ such that $s_{\tau }s_{[j\;/\;m]}x\in F$.}

\bigskip

As is known, neat embeddability into a class with $\varepsilon $ extra
dimensions, where $\varepsilon $ is infinite, implies neat embeddability into
the class with any infinitely many extra dimensions (\cite{HMT1} I.
2.6.35). In otherwords, strectching extra dimensions to $\omega$ is enough to do any infinite strectching using arbirary more dimensions.
Therefore, from now on, we may assume that $\varepsilon >$ $\max
(\alpha ,\left| A\right| ),$ and $\varepsilon +\alpha $ is a regular cardinal.

\begin{lemma}\textit{Let }$a$\textit{\ be an arbitrary, but fixed
non-zero element of }$A$\textit{\ and} $\varepsilon >$ max ($\alpha ,\left|
A\right| $), where $\varepsilon +\alpha $ is regular, and assume that $%
\mathcal{A}\in $ S$Nr_{\alpha }\mathcal{B}$\textit{\ for some} $\mathcal{B}$ 
$\in $ $_{m}$\CPE$_{\alpha +\varepsilon }^{-}.$ \textit{Then, there exists a
proper Boolean filter }$\mathcal{D}$ \textit{in }$\mathcal{B}$,\textit{\
such that }$a\in D$\textit{\ and any arbitrary ultrafilter containing }$D$%
\textit{\ is a perfect ultrafilter in} $\mathcal{B}$.
\end{lemma}
\begin{proof} Henkin's completeness proof is appropriately adapted to the algebras in question.
We use freely the axiomatization $(CP_0)-(E_3)$ in \cite{Fer} p.156.
Let
\begin{equation*}
X=\left\{ s_{\tau }c_{j}x:\tau \in adm,\;j\in \alpha ,x\in A\right\} .
\end{equation*}
Let $\beta $ denote the ordinal $\alpha +\varepsilon .$ $\left| X\right|
\leq $ $^{m}\beta \cdot \alpha \cdot \left| A\right| .$ $\beta >\alpha ,$ $%
\beta $ is regular, therefore, $^{m}\beta =\beta .$ $\beta >\max $ $(\alpha
,|A|)$ and $^{m}\beta =\beta $ imply that $|X|\leq \beta $. Let $\rho
\,:\,\beta \rightarrow X$ be a fixed enumeration of $X$.

Let $F_{0}$ be the Boolean (BA) filter of $\mathcal{B}$ generated by $a$. We
define $\mathit{recursively}$ an increasing sequence $\left\langle
F_{i}\,:\,i<\beta \right\rangle $ of proper BA filters in $\mathcal{B}$.

Assume that $\rho _{1}=s_{\tau ^{\prime }}c_{j^{\prime }}^{{}}x^{\prime },$
where $\tau ^{\prime }\in adm,$ $j^{\prime }\in \alpha $ and $x^{\prime }\in
A.$ Let $F_{1}$ be the filter generated by the set $G_{1}$ in $\mathcal{B}$,
where

\begin{equation*}
G_{1}=F_{0}\cup \left\{ s_{\tau ^{\prime }}c_{j^{\prime }}^{{}}x^{\prime
}\rightarrow s_{\tau ^{\prime }}^{{}}s_{[j^{\prime }\text{ / }m^{\prime
}]}x^{\prime }\right\}
\end{equation*}
\newline
and $m^{\prime }\in \beta $ is such that $m^{\prime }\notin \;$Rg $\tau
^{\prime }\cup \alpha .$ We will show that $F_{1}$ is a proper filter in $%
\mathcal{B}.$

Let $n$ be a fixed ordinal ($n<\beta $). Assume that $F_{i}$ $(0\leq i<n-1$)
has been defined by the fixed generator system $G_{i}$ ($G_{0}\subset
G_{1}\subset \ldots \subset G_{n-1}).$ Let $\rho _{n}=s_{\tau }c_{j}x$,
where $\tau \in adm$, $j\in \alpha $ and $x\in A$.

If $n$ is a \textit{successor }ordinal, let $F_{n}$ be the filter in $%
\mathcal{B}$ generated by the set 
\begin{equation}
G_{n}=G_{n-1}\cup \{s_{\tau }c_{j}x\rightarrow s_{\tau
}s_{[j\,/\,\,m_{n}]}x\}  \label{*1}
\end{equation}
\qquad \newline
where $m_{n}\in \beta $ is such an ordinal that $m_{n}\notin \alpha $ $\cup $
$\bigcup\limits_{g\in G_{n-1}}$dim $g$ and $\tau m_{n}=m_{n}.$ Such an
ordinal exists because $\left| \text{dim }g\right| \leq m$, $\left|
G_{n-1}\right| <n$, hence $\left| \bigcup\limits_{g\in G_{n-1}}\text{dim }%
g\right| <m\cdot n<\beta ,$ furthermore $\alpha <\beta ,$ too. Let us denote
by $g_{n}$ the generator element $s_{\tau }c_{j}^{{}}x\rightarrow s_{\tau
}^{{}}s_{[j\text{ / }m]}x.$

If $n$ is \textit{a limit }ordinal, then let $F_{n}$ be $\bigcup%
\limits_{i<n}F_{i}$. Obviously, $F_{j}\subseteq F_{k}$ if $j<k$.

It remains to show that the filter $F_{n}$ generated by the set in (\ref{*1}%
) is a \textit{proper} filter. The only case worthwhile considering is the
case when $n$ is a successor ordinal. Indirectly, assume that $F_{n-1}$ is
proper and assume, seeking a contradiction, that $F_{n}$ is not.

In what follows, let us denote $m_{n}$ by $m$, for short. Suppose on the
contrary that $-(s_{\tau }c_{j}x\rightarrow s_{\tau }s_{[j\text{ }/\,\text{\ 
}m]}x)$ belongs to $F_{n}$. The property of generating filters in BA's
implies that there are finitely many generators in $F_{n-1}$ such that 
\begin{equation*}
a\ \cdot (s_{\tau _{1}}c_{j_{1}}x_{1}\rightarrow s_{\tau _{1}}s_{[j_{1}\text{
}/\,\text{\ }m_{1}]}x_{1})\cdot \ldots \cdot (s_{\tau
_{k}}c_{j_{k}}x_{k}\rightarrow s_{\tau _{k}}s_{[j_{k}\text{ }/\,\text{\ }%
m_{k}]}x_{k})\leq
\end{equation*}
\begin{equation}
\leq -(s_{\tau }c_{j}x\rightarrow s_{\tau }s_{[j\text{ }/\,\text{\ }%
m]}^{{}}x)  \label{eq15}
\end{equation}
where $x_{1},x_{2},\ldots ,x_{k},x$ are in $A$. Let us apply $%
c_{m}^{\partial }$ to both sides of this inequality ($c_{m}^{\partial }$
denotes the operator $-c_{m}-).$

If $x$ is any factor of the left-hand side, then the conditions $m\notin $%
dim $g$, $g\in G_{n-1},$ $x\in F_{n-1}$ and (\ref{felcserélhetoB}) imply
that 
\begin{equation}
c_{m}(s_{\tau _{i}}c_{j_{i}}x_{i}\rightarrow s_{\tau _{i}}s_{[j_{1}\text{ }%
/\,\text{\ }m_{1}]}x)=s_{\tau _{i}}c_{j_{i}}x_{i}\rightarrow s_{\tau
_{i}}s_{[j_{i}\text{ }/\,\text{\ }m_{i}]}x.  \label{eq16}
\end{equation}
But (\ref{eq16}) is true for $c_{m}^{\partial }$ instead of $c_{m}$, using
that $c_{m}(-c_{m}x)=-c_{m}x$, $x\in B$. Thus, applying $c_{m}^{\partial }$
to the left-hand side of (\ref{eq15}) does not change it and it must be
different from $0$ because $F_{n-1}$ is a proper filter. Here we have used
that $c_{m}^{\partial }(u+v)=c_{m}^{\partial }u+c_{m}^{\partial }v$, which
is a consequence of (CP3).

Applying $c_{m}^{\partial }$ to the right-hand side of (\ref{eq15}), we can
show that we obtain zero. We have 
\begin{align}
c_{m}^{\partial }(-(s_{\tau }c_{j}x\rightarrow s_{\tau }s_{[j\text{ }/\,%
\text{\ }m]}x))& =-c_{m}[-s_{\tau }c_{j}x+s_{\tau }s_{[j\text{ }/\,\text{\ }%
m]}x)]  \notag \\
& =-[c_{m}(-s_{\tau }c_{j}x)+c_{m}s_{\tau }s_{[j\text{ }/\,\text{\ }m]}x]
\label{eq21}
\end{align}
because $c_{m}(u+v)=c_{m}u+c_{m}v$.

On the one hand, as regards $c_{m}(-s_{\tau }c_{j}x)$ in (\ref{eq21}), by $%
m\notin \alpha ,$ $\tau m=m\ $ and (\ref{felcserélhetoB}),

\begin{equation}
c_{m}(-s_{\tau }c_{j}x)=c_{m}(-c_{m}s_{\tau }c_{j}x)=-c_{m}s_{\tau }c_{j}x.
\label{eq21.5}
\end{equation}
Here $c_{m}s_{\tau }c_{j}x=s_{\tau }c_{m}s_{[j\text{ }/\,\text{\ }m]}c_{m}x$
because

\begin{equation}
c_{j}x=c_{m}s_{[j\text{ }/\,\text{\ }m]}x  \label{csere}
\end{equation}
\newline
by (\ref{felcserélhetoB}) if $x\in A$. Therefore, 
\begin{equation}
c_{m}(-s_{\tau }c_{j}x)=-s_{\tau }c_{m}s_{[j\text{ }/\,\text{\ }m]}c_{m}x\,
\label{eq22}
\end{equation}

On the other hand, as regards $c_{m}s_{\tau }s_{[j\text{ }/\,\text{\ }m]}x,$
i.e., $c_{m}s_{\tau }s_{[j\text{ }/\,\text{\ }m]}c_{m}x$ in (\ref{eq21})$,$ 
\begin{equation}
c_{m}s_{\tau }s_{[j\text{ }/\,\text{\ }m]}c_{m}x=s_{\tau }c_{m}s_{[j\text{ }%
/\,\text{\ }m]}c_{m}x.  \label{eq23}
\end{equation}
\newline
by (\ref{felcserélhetoA}), where $m\notin $ $\alpha $,$\tau m=m$.

From (\ref{eq22}) and (\ref{eq23}) we get that (\ref{eq21}) is zero. It is a
contradiction because the left-hand side of (\ref{eq21}) is different from
zero. Therefore we have shown that, in fact, $F_{n}$ is a proper filter.

Now, we have a sequence $G_{0}=F_{0}\subset F_{1}\subset F_{2}\subset \ldots
\subset F_{n}\subset \ldots $ of proper filters. Now let

\begin{equation}
\mathcal{D}=\bigcup\limits_{n}\{F_{n}\,:\,n<\beta \}.  \label{filterdef}
\end{equation}

$\mathcal{D}$ is a proper filter, too. $\mathcal{D}$ contains all the
elements of the form $s_{\tau }c_{j}x,$ $x\in A.$ It is easily seen that $%
\mathcal{D}$ is the desired filter.

\bigskip

But we are not quite finished. We need an extra technical  trick to make things tic. 

\textit{The required perfect ultrafilter }$F$ \textit{in }$\mathcal{B}$%
, \textit{extending the filter} $\mathcal{D}$ will be 
defined as follows:

\bigskip
Take the {\it minimal  completion} $\mathcal{B}$' of $\mathcal{B}
$ (see \cite{HMT2}, I., 2.7.2). exists because operations are additive (here diagonals play an essential role; for infact our variety is conjugated).
Take 
the filter $F^{\prime }$ in $\mathcal{B}$', generated by (the generators of) $%
\mathcal{D}$ -- such a filter $F^{\prime }$ obviously exists. Next consider any
fixed ultrafilter ($F^{\prime }$)$^{+}$ in $\mathcal{B}$', which extends $%
F^{\prime }.$ The restriction $F$ of ($F^{\prime }$)$^{+}$ to $%
\mathcal{B}$ is an ultrafilter in $\mathcal{B}$, which we take for the required extension of the filter $\mathcal{D}$ in 
$\mathcal{B}$. 
\end{proof}

Now using the ideas above, let us prove interpolation. Let $\A$ be the free algebra, $X_1, X_2\subseteq \A$, $a\in \Sg X_1$ and $b\in \Sg X_2$.
such that that $a\leq b$. Assume that an interpolant {\it does not} exist. 
Neatly embed the free algebra  into the {\it full} neat reduct of a dilation having enough spare dimensions, the dimension of the dilation should also
be a regular cardinal as in the above proof.
This gives enough space to maneuve, in the processs, eliminating  cylindrifiers.
This part is easy, because we have so many substitutions. 

The second part consists of the construction the two perfect ultrafilters, 
the first contains $a$, the second contains $-c$, constructed in such a way that they agree on $\Sg^{\A}(X_1\cap X_2)$, 
then  one constructs yet a third perfect ultarfilter in the last algebra, and obtains (by Ferenczi's result) 
a relativized representation $h$ of $\A$ using its freeness, on the set of generators $X_1\cup X_2$
such that $h(a.-c)\neq 0$, 
but this is a contradiction.

Let us implement the above sketch, but we consider the most basic cylindric polyadic algebras, namely 
the class  $_m\CPEA.$ The rest of the cases for non-commutive cylindric polyadic 
algebras, are entirely analagous.

Ferenczi proves that if we have a neat embedding into enough infinite spare dimensions, we have representability. 
The next lemma shows that we can always neatly embed our algebras
in enough spare dimension. The idea is that our algebras form what Diagneault and Monk call transformation 
systems \cite{DM}.
Strictly speacking this applies to their reduct obtained by discarding all operations except {\it all} substitutions.

\begin{lemma}\label{net} Let $\A$ be a cylindric polyadic algebra of dimension $\alpha$. 
Then for every $\beta>\alpha$ there exists a cylindric polyadic algebra of dimension $\beta$ 
such that
$\A\subseteq \Nr_{\alpha}\B$, and furthermore, for all $X\subseteq \A$ we have
$$\Sg^{\A}X=\Sg^{\Nr_{\alpha}\A}X=\Nr_{\alpha}\Sg^{\B}X.$$ In particular, $\A=\Nr_{\alpha}\B$.
$\Sg^{\B}\A$ is called the minimal dilation of $\A.$ 
\end{lemma}
\begin{demo}{Proof} The proof depends essentially on the abundance of substitutions; we have all of them, which makes stretching dimensions 
possible. We provide a proof for cylindric polyadic algebras; the rest of the cases are like the corresponding prof in \cite{DM} for Boolean polyadic 
algebras. 

We extensively use the techniques in \cite{DM}, but we have to watch out, for we only have finite cylindrifications.
Let $(\A, \alpha,S)$ be a transformation system. 
That is to say, $\A$ is a Boolean algebra and $S:{}^\alpha\alpha\to End(\A)$ is a homomorphism. For any set $X$, let $F(^{\alpha}X,\A)$ 
be the set of all functions from $^{\alpha}X$ to $\A$ endowed with Boolean operations defined pointwise and for 
$\tau\in {}^\alpha\alpha$ and $f\in F(^{\alpha}X, \A)$, ${\sf s}_{\tau}f(x)=f(x\circ \tau)$. 
This turns $F(^{\alpha}X,\A)$ to a transformation system as well. 
The map $H:\A\to F(^{\alpha}\alpha, \A)$ defined by $H(p)(x)={\sf s}_xp$ is
easily checked to be an isomorphism. Assume that $\beta\supseteq \alpha$. Then $K:F(^{\alpha}\alpha, \A)\to F(^{\beta}\alpha, \A)$ 
defined by $K(f)x=f(x\upharpoonright \alpha)$ is an isomorphism. These facts are straighforward to establish, cf. theorem 3.1, 3.2 
in \cite{DM}. 
$F(^{\beta}\alpha, \A)$ is called a minimal dilation of $F(^{\alpha}\alpha, \A)$. Elements of the big algebra, or the cylindrifier free 
dilation, are of form ${\sf s}_{\sigma}p$,
$p\in F(^{\beta}\alpha, \A)$ where $\sigma$ is one to one on $\alpha$, cf. \cite{DM} theorem 4.3-4.4.

We say that $J\subseteq I$ supports an element $p\in A,$ if whenever $\sigma_1$ and  $\sigma_2$ are 
transformations that agree on $J,$ then  ${\sf s}_{\sigma_1}p={\sf s}_{\sigma_2}p$.
$\Nr_JA$, consisting of the elements that $J$ supports, is just the  neat $J$ reduct of $\A$; 
with the operations defined the obvious way as indicated above. 
If $\A$ is an $\B$ valued $I$ transformaton system with domain $X$, 
then the $J$ compression of $\A$ is isomorphic to a $\B$ valued $J$ transformation system
via $H: \Nr_J\A\to F(^JX, \A)$ by setting for $f\in\Nr_J\A$ and $x\in {}^JX$, $H(f)x=f(y)$ where $y\in X^I$ and $y\upharpoonright J=x$, 
cf. \cite{DM} theorem 3.10.

Now let $\alpha\subseteq \beta.$ If $|\alpha|=|\beta|$ then the the required algebra is defined as follows. 
Let $\mu$ be a bijection from $\beta$ onto $\alpha$. For $\tau\in {}^{\beta}\beta,$ let ${\sf s}_{\tau}={\sf s}_{\mu\tau\mu^{-1}}$ 
and for each $i\in \beta,$ let 
${\sf c}_i={\sf c}_{\mu(i)}$. Then this defined $\B\in GPHA_{\beta}$ in which $\A$ neatly embeds via ${\sf s}_{\mu\upharpoonright\alpha},$
cf. \cite{DM} p.168.  Now assume that $|\alpha|<|\beta|$.
Let $\A$ be a  given polyadic algebra of dimension $\alpha$; discard its cylindrifications and then take its minimal dilation $\B$, 
which exists by the above.
We need to define cylindrifications on the big algebra, so that they agree with their values in $\A$ and to have $\A\cong \Nr_{\alpha}\B$. We let (*):
$${\sf c}_k{\sf s}_{\sigma}^{\B}p={\sf s}_{\rho^{-1}}^{\B} {\sf c}_{\rho(\{k\}\cap \sigma \alpha)}{\sf s}_{(\rho\sigma\upharpoonright \alpha)}^{\A}p.$$
Here $\rho$ is a any permutation such that $\rho\circ \sigma(\alpha)\subseteq \sigma(\alpha.)$
Then we claim that the definition is sound, that is, it is independent of $\rho, \sigma, p$. 
Towards this end, let $q={\sf s}_{\sigma}^{\B}p={\sf s}_{\sigma_1}^{\B}p_1$ and 
$(\rho_1\circ \sigma_1)(\alpha)\subseteq \alpha.$

We need to show that (**)
$${\sf s}_{\rho^{-1}}^{\B}{\sf c}_{[\rho(\{k\}\cap \sigma(\alpha)]}^{\A}{\sf s}_{(\rho\circ \sigma\upharpoonright \alpha)}^{\A}p=
{\sf s}_{\rho_1{^{-1}}}^{\B}{\sf c}_{[\rho_1(\{k\}\cap \sigma(\alpha)]}^{\A}{\sf s}_{(\rho_1\circ \sigma\upharpoonright \alpha)}^{\A}p.$$
Let $\mu$ be a permutation of $\beta$ such that
$\mu(\sigma(\alpha)\cup \sigma_1(\alpha))\subseteq \alpha$.
Now applying ${\sf s}_{\mu}$ to the left hand side of (**), we get that 
$${\sf s}_{\mu}^{\B}{\sf s}_{\rho^{-1}}^{\B}{\sf c}_{[\rho(\{k\})\cap \sigma(\alpha)]}^{\A}{\sf s}_{(\rho\circ \sigma|\alpha)}^{\A}p
={\sf s}_{\mu\circ \rho^{-1}}^{\B}{\sf c}_{[\rho(\{k\})\cap \sigma(\alpha)]}^{\A}{\sf s}_{(\rho\circ \sigma|\alpha)}^{\A}p.$$
The latter is equal to
${\sf c}_{(\mu(\{k\})\cap \sigma(\alpha))}{\sf s}_{\sigma}^{\B}q.$
Now since $\mu(\sigma(\alpha)\cap \sigma_1(\alpha))\subseteq \alpha$, we have
${\sf s}_{\mu}^{\B}p={\sf s}_{(\mu\circ \sigma\upharpoonright \alpha)}^{\A}p={\sf s}_{(\mu\circ \sigma_1)\upharpoonright \alpha)}^{\A}p_1\in A$.
It thus follows that 
$${\sf s}_{\rho^{-1}}^{\B}{\sf c}_{[\rho(\{k\})\cap \sigma(\alpha)]}^{\A}{\sf s}_{(\rho\circ \sigma\upharpoonright \alpha)}^{\A}p=
{\sf c}_{[\mu(\{k\})\cap \mu\circ \sigma(\alpha)\cap \mu\circ \sigma_1(\alpha))}{\sf s}_{\sigma}^{\B}q.$$ 
By exactly the same method, it can be shown that 
$${\sf s}_{\rho_1{^{-1}}}^{\B}{\sf c}_{[\rho_1(\{k\})\cap \sigma(\alpha)]}^{\A}{\sf s}_{(\rho_1\circ \sigma\upharpoonright \alpha)}^{\A}p
={\sf c}_{[\mu(\{k\})\cap \mu\circ \sigma(\alpha)\cap \mu\circ \sigma_1(\alpha))}{\sf s}_{\sigma}^{\B}q.$$ 
By this we have proved (**).

Furthermore, it defines the required algebra $\B$. Let us check this.
Since our definition is slightly different than that in \cite{DM}, by restricting cylindrifications to be olny finite, 
we need to check the polyadic axioms which is tedious but routine. The idea is that every axiom can be pulled back to 
its corresponding axiom holding in the small algebra 
$\A$.
We check only the axiom $${\sf c}_k(q_1\land {\sf c}_kq_2)={\sf c}_kq_1\land {\sf c}_kq_2.$$
We follow closely \cite{DM} p. 166. 
Assume that $q_1={\sf s}_{\sigma}^{\B}p_1$ and $q_2={\sf s}_{\sigma}^{\B}p_2$. 
Let $\rho$ be a permutation of $I$ such that $\rho(\sigma_1I\cup \sigma_2I)\subseteq I$ and let 
$$p={\sf s}_{\rho}^{\B}[q_1\land {\sf c}_kq_2].$$
Then $$p={\sf s}_{\rho}^{\B}q_1\land {\sf s}_{\rho}^{\B}{\sf c}_kq_2
={\sf s}_{\rho}^{\B}{\sf s}_{\sigma_1}^{\B}p_1\land {\sf s}_{\rho}^{\B}{\sf c}_k {\sf s}_{\sigma_2}^{\B}p_2.$$
Now we calculate ${\sf c}_k{\sf s}_{\sigma_2}^{\B}p_2.$
We have by (*)
$${\sf c}_k{\sf s}_{\sigma_2}^{\B}p_2= {\sf s}^{\B}_{\sigma_2^{-1}}{\sf c}_{\rho(\{k\}\cap \sigma_2I)} {\sf s}^{\A}_{(\rho\sigma_2\upharpoonright I)}p_2.$$
Hence $$p={\sf s}_{\rho}^{\B}{\sf s}_{\sigma_1}^{\B}p_1\land {\sf s}_{\rho}^{\B}{\sf s}^{\B}_{\sigma^{-1}}{\sf c}_{\rho(\{k\}\cap \sigma_2I)} 
{\sf s}^{\A}_{(\rho\sigma_2\upharpoonright I)}p_2.$$
\begin{equation*}
\begin{split}
&={\sf s}^{\A}_{\rho\sigma_1\upharpoonright I}p_1\land {\sf s}_{\rho}^{\B}{\sf s}^{\A}_{\sigma^{-1}}{\sf c}_{\rho(\{k\}\cap \sigma_2I)} 
{\sf s}^{\A}_{(\rho\sigma_2\upharpoonright I)}p_2,\\
&={\sf s}^{\A}_{\rho\sigma_1\upharpoonright I}p_1\land {\sf s}_{\rho\sigma^{-1}}^{\A}
{\sf c}_{\rho(\{k\}\cap \sigma_2I)} {\sf s}^{\A}_{(\rho\sigma_2\upharpoonright I)}p_2,\\
&={\sf s}^{\A}_{\rho\sigma_1\upharpoonright I}p_1\land {\sf c}_{\rho(\{k\}\cap \sigma_2I)} {\sf s}^{\A}_{(\rho\sigma_2\upharpoonright I)}p_2.\\
\end{split}
\end{equation*} 
Now $${\sf c}_k{\sf s}_{\rho^{-1}}^{\B}p={\sf c}_k{\sf s}_{\rho^{-1}}^{\B}{\sf s}_{\rho}^{\B}(q_1\land {\sf c}_k q_2)={\sf c}_k(q_1\land {\sf c}_kq_2)$$
We next calculate ${\sf c}_k{\sf s}_{\rho^{-1}}p$.
Let $\mu$ be a permutation of $I$ such that $\mu\rho^{-1}I\subseteq I$. Let $j=\mu(\{k\}\cap \rho^{-1}I)$.
Then applying (*), we have:
\begin{equation*}
\begin{split}
&{\sf c}_k{\sf s}_{\rho^{-1}}p={\sf s}^{\B}_{\mu^{-1}}{\sf c}_{j}{\sf s}_{(\mu\rho^{-1}|I)}^{\A}p,\\
&={\sf s}^{\B}_{\mu^{-1}}{\sf c}_{j}{\sf s}_{(\mu\rho^{-1}|I)}^{\A}
{\sf s}^{\A}_{\rho\sigma_1\upharpoonright I}p_1\land {\sf c}_{(\rho\{k\}\cap \sigma_2I)} {\sf s}^{\B}_{(\rho\sigma_2\upharpoonright I)}p_2,\\
 &={\sf s}^{\B}_{\mu^{-1}}{\sf c}_{j}[{\sf s}_{\mu \sigma_1\upharpoonright I}p_1\land r].\\
\end{split}
\end{equation*}

where 
$$r={\sf s}_{\mu\rho^{-1}}^{\B}{\sf c}_j {\sf s}_{\rho \sigma_2\upharpoonright I}^{\A}p_2.$$
Now ${\sf c}_kr=r$. Hence, applying the axiom in the small algebra, we get: 
$${\sf s}^{\B}_{\mu^{-1}}{\sf c}_{j}[{\sf s}_{\mu \sigma_1\upharpoonright I}^{\A}p_1]\land {\sf c}_k q_2
={\sf s}^{\B}_{\mu^{-1}}{\sf c}_{j}[{\sf s}_{\mu \sigma_1\upharpoonright I}^{\A}p_1\land r].$$
But
$${\sf c}_{\mu(\{k\}\cap \rho^{-1}I)}{\sf s}_{(\mu\sigma_1|I)}^{\A}p_1=
{\sf c}_{\mu(\{k\}\cap \sigma_1I)}{\sf s}_{(\mu\sigma_1|I)}^{\A}p_1.$$
So 
$${\sf s}^{\B}_{\mu^{-1}}{\sf c}_{k}[{\sf s}_{\mu \sigma_1\upharpoonright I}^{\A}p_1]={\sf c}_kq_1,$$ and 
we are done. 


For the second part, let $\A\subseteq \Nr_{\alpha}\B$ and $A$ generates $\B$ then $\B$ consists of all elements ${\sf s}_{\sigma}^{\B}x$ such that 
$x\in A$ and $\sigma$ is a transformation on $\beta$ such that
$\sigma\upharpoonright \alpha$ is one to one \cite{DM} theorem 3.3 and 4.3.
Now suppose $x\in \Nr_{\alpha}\Sg^{\B}X$ and $\Delta x\subseteq
\alpha$. There exists $y\in \Sg^{\A}X$ and a transformation $\sigma$
of $\beta$ such that $\sigma\upharpoonright \alpha$ is one to one
and $x={\sf s}_{\sigma}^{\B}.$  
Let $\tau$ be a 
transformation of $\beta$ such that $\tau\upharpoonright  \alpha=Id
\text { and } (\tau\circ \sigma) \alpha\subseteq \alpha.$ Then
$x={\sf s}_{\tau}^{\B}x={\sf s}_{\tau}^{\B}{\sf s}_{\sigma}y=
{\sf s}_{\tau\circ \sigma}^{\B}y={\sf s}_{\tau\circ
\sigma\upharpoonright \alpha}^{\A'}y.$
\end{demo}
From now on cylindric polyadic algebras are denoted by $_m\CPEA$.

\begin{theorem} Let $\beta$ be a cardinal, and $\A=\Fr_{\beta}{}_m\CPEA_{\alpha}$ be the free algebra on $\beta$ generators.
Let $X_1, X_2\subseteq
\beta$, $a\in \Sg^{\A}X_1$ and $c\in \Sg^{\A}X_2$ be such that $a\leq c$.
Then there exists $b\in \Sg^{\A}(X_1\cap X_2)$ such that $a\leq b\leq c.$
\end{theorem}

\begin{proof}  Let $a\in \Sg^{\A} X_1$ and $c\in \Sg^{\A}X_2$ be such that $a\leq c$. 
We want to find an interpolant in 
$\Sg^{\A}(X_1\cap X_2)$. Assume that $\mu$ is a regular cardinal $>max(|\alpha|,|A|)$.
Let $\B\in {}_m\CPEA_{\mu}$, such that $\A=\Nr_{\alpha}\B$, 
and $A$ generates $\B$.   
Let $H_{\mu}=\{\rho\in {}^{\mu}\mu: |\rho(\alpha)\cap (\mu\sim \alpha)|<\omega\}$.
Let $S$ be the semigroup generated by $H_{\mu}.$ 
Let $\B'\in {}_m\CPEA_{\mu}$ be an ordinary  dilation of $\A$ where all transformations in $^{\mu}\mu$ are used.  
(This can be easily defined like in the case of ordinary polyadic algebras). Then $\A=\Nr_{\alpha}\B'$. 
We take a suitable reduct of $\B'$. Let $\B$ be the subalgebra of 
$\B'$ generated from $A$ be all operations except for substitutions indexed by transformations not in $S$.
Then, of course $A\subseteq \B$; in fact, $\A=\Nr_{\alpha}\B$, since for each $\tau\in {}^{\alpha}\alpha$, $\tau\cup Id\in S.$ 
Then one can  show inductively that for $b\in B$, if 
$|\Delta b\sim \alpha|<\omega$, and $\rho\in S$, then $|\rho(\Delta b)\sim \alpha|<\omega$. 
Next one defines filters in the dilations $\Sg^{\B}X_1$ and in $\Sg^{\B}X_2$
like in Ferenczi \cite{Fer}, but they have to be compatible on the common subalgebras. This needs some work.
Assume that no interpolant exists in $\Sg^{\A}(X_1\cap X_2)$. 
Then no interpolant exists in $\Sg^{\B}(X_1\cap X_2)$, for if one does, 
then it can be pulled back, using the first part, by cylindrifiers only on fiinitely many indices, to an interpolant
in $\Sg^{\A}(X_1\cap X_2)$, which we assume does not exists. 
We eventually arrive at a contradiction.
Arrange $adm\times\mu \times \Sg^{\B}(X_1)$
and $adm\times\mu\times \Sg^{\B}(X_2)$
into $\omega$-termed sequences:
$$\langle (\tau_i,k_i,x_i): i\in \mu\rangle\text {  and  }\langle (\sigma_i,l_i,y_i):i\in \mu\rangle
\text {  respectively.}$$
is as desired. 
Thus we can define by recursion (or step-by-step)
$\mu$-termed sequences of witnesses:
$$\langle u_i:i\in \mu\rangle \text { and }\langle v_i:i\in \mu\rangle$$
such that for all $i\in \mu$ we have:
$$u_i\in \mu\smallsetminus
(\Delta a\cup \Delta c)\cup \cup_{j\leq i}(\Delta x_j\cup \Delta y_j\cup Do\tau_j\cup Rg\tau_j\cup Do\sigma_j\cup Rg\sigma_j)\cup \{u_j:j<i\}\cup \{v_j:j<i\}$$
and
$$v_i\in \mu\smallsetminus(\Delta a\cup \Delta c)\cup
\cup_{j\leq i}(\Delta x_j\cup \Delta y_j\cup Do\tau_j\cup Rg\tau_j\cup Do\sigma_j\cup Rg\sigma_j))\cup \{u_j:j\leq i\}\cup \{v_j:j<i\}.$$
For a cylindric algebra $\cal D$ we write $Bl\cal D$ to denote its boolean reduct.
For $i, j<\mu$, $i\neq j$,
$s_i^jx=c_i(d_{ij}\cdot x)$ and $s_i^ix=x$. $s_i^j$ is a unary operation
that abstracts the operation of substituting the
variable $v_i$ for the variable $v_j$ such that
the substitution is free.
For a boolean algebra $\cal C$  and $Y\subseteq \cal C$, we write
$fl^{\cal C}Y$ to denote the boolean filter generated by $Y$ in $\cal C.$  Now let
$$Y_1= \{a\}\cup \{-{\sf s}_{\tau_i}{\sf  c}_{k_i}x_i+{\sf s}_{\tau_i}{\sf s}_{u_i}^{k_i}x_i: i\in \mu\},$$
$$Y_2=\{-c\}\cup \{-{\sf s}_{\sigma_i}{\sf  c}_{l_i}y_i+{\sf s}_{\sigma_i}{\sf s}_{v_i}^{l_i}y_i:i\in \mu\},$$
$$H_1= fl^{Bl\Sg^{B}(X_1)}Y_1,\  H_2=fl^{Bl\Sg^B(X_2)}Y_2,$$ and
$$H=fl^{Bl\Sg^{B}(X_1\cap X_2)}[(H_1\cap \Sg^{B}(X_1\cap X_2)
\cup (H_2\cap \Sg^{B}(X_1\cap X_2)].$$
We claim that $H$ is a proper filter of $\Sg^{B}(X_1\cap X_2).$
To prove this it is sufficient to consider any pair of finite, strictly
increasing sequences of natural numbers
$$\eta(0)<\eta(1)\cdots <\eta(n-1)<\mu\text { and } \xi(0)<\xi(1)<\cdots
<\xi(m-1)<\mu,$$
and to prove that the following condition holds:

(1) For any $b_0$, $b_1\in \Sg^{B}(X_1\cap X_2)$ such that
$$a.\prod_{i<n}(-{\sf s}_{\tau_{\eta(i)}}{\sf  c}_{k_{\eta(i)}}x_{\eta(i)}+{\sf s}_{\tau_{\eta(i)}}{\sf s}_{u_{\eta(i)}}^{k_{\eta(i)}}x_{\eta(i)})\leq b_0$$
and
$$(-c).\prod_{i<m}
(-{\sf s}_{\sigma_{\xi(i)}}{\sf  c}_{l_{\xi(i)}}y_{\xi(i)}+{\sf s}_{\sigma_{\xi(i)}}{\sf s}_{v_{\xi(i)}}^{l_{\xi(i)}}y_{\xi(i)})\leq b_1$$
we have
$$b_0.b_1\neq 0.$$
We prove this by induction on $n+m$.
If $n+m=0$, then (1) simply
expresses the fact that no interpolant of $a$ and $c$ exists in
$\Sg^{\B}(X_1\cap X_2).$
In more detail: if $n+m=0$, then $a_0\leq b_0$
and $-c\leq b_1$. So if $b_0.b_1=0$, we get $a\leq b_0\leq -b_1\leq c.$
Now assume that $n+m>0$ and for the time being suppose that $\eta(n-1)>\xi(m-1)$.
Apply ${\sf  c}_{u_{\eta(n-1)}}$ to both sides of the first inclusion of (1).
By $u_{\eta(n-1)}\notin \Delta a$, i.e. ${\sf  c}_{u_{\eta(n-1)}}a=a$,
and by recalling that ${\sf  c}_i({\sf  c}_ix.y)={\sf  c}_ix.{\sf  c}_iy$, we get (2)
$$a.{\sf  c}_{u_{\eta(n-1)}}\prod_{i<n}(-{\sf s}_{\tau_{\eta(i)}}{\sf  c}_{k_{\eta(i)}}x_{\eta(i)}+{\sf s}_{\tau_{\eta(i)}}{\sf s}_{u_{\eta(i)}}^{k_{\eta(i)}}x_{\eta(i)})\leq {\sf  c}_{u_{\eta(n-1)}}b_0.$$
Let ${\sf  c}_i^{\partial}(x)=-{\sf  c}_i(-x)$.  ${\sf  c}_i^{\partial}$
is the algebraic counterpart of the universal quantifier $\forall x_i$.
Now apply ${\sf  c}_{u_{\eta(n-1)}}^{\partial}$ to the second inclusion of (1).
By noting that ${\sf  c}_i^{\partial}$, the dual of ${\sf  c}_i$,
distributes over the boolean meet and by
$u_{\eta(n-1)}\notin \Delta c=\Delta (-c)$ we get (3)
$$(-c).\prod_{j<m}{\sf  c}_{u_{\eta(n-1)}}^{\partial}(-{\sf s}_{\sigma_{\xi(i)}}{\sf  c}_{l_{\xi(i)}}y_{\xi(i)}+{\sf s}_{\sigma_{\xi(i)}}{\sf s}_{v_{\xi(i)}}^{l_{\xi(i)}}y_{\xi(i)})\leq {\sf  c}_{u_{\eta(n-1)}}^{\partial}b_1.$$
Before going on, we formulate (and prove) a claim that will enable us to eliminate
the quantifier ${\sf  c}_{u_{\eta(n-1)}}$ (and its dual) from (2) (and (3)) above.

For the sake of brevity set for each $i<n$ and each $j<m:$
$$z_i=-{\sf  c}_{k_{\eta(i)}}x_{\eta(i)}+{\sf s}_{u_{\eta(i)}}^{k_{\eta(i)}}x_{\eta(i)}$$  and
$$t_i=-{\sf  c}_{l_{\xi(i)}}y_{\xi(i)}+{\sf s}_{v_{\xi(i)}}^{l_{\xi(i)}}y_{\xi(i)}.$$
Then $(i)$ and $(ii)$ below hold:
$$(i)\  {\sf  c}_{u_{\eta(n-1)}}z_i=z_i\text { for }i<n-1 \text { and }{\sf  c}_{u_{\eta(n-1)}}z_{n-1}=1.$$
$$(ii)\ {\sf  c}_{u_{\eta(n-1)}}^{\partial}t_j=t_j\text { for all }j<m.$$

 {\it Proof of ${\sf  c}_{u_{\eta_{n-1}}}z_i=z_i$ for  $i<n-1$.}

Let $i<n-1$. Then by the choice of witnesses we have
$$u_{\eta(n-1)}\neq u_{\eta(i)}.$$
Also it is easy to see that for all $i,j\in \alpha$ we have
$$\Delta {\sf  c}_jx\subseteq \Delta x\text {  and that }
\Delta {\sf s}_j^ix\subseteq \Delta x\smallsetminus\{i\}\cup \{j\},$$
In particular,
$$u_{\eta(n-1)}\notin \Delta {\sf  c}_{k_{\eta(i)}}x_{\eta(i)}\text { and }
u_{\eta(n-1)}\notin \Delta ({\sf s}_{u_{\eta(i)}}^{k_{\eta(i)}}x_{\eta(i)}).$$
It thus follows that
$${\sf  c}_{u_{\eta(n-1)}}(-{\sf  c}_{k_{\eta(i)}}x_{\eta(i)})=-{\sf  c}_{k_{\eta(i)}}x_{\eta(i)}\text { and }
{\sf  c}_{u_{\eta(n-1)}} ({\sf s}_{u_{\eta(i)}}^{k_{\eta(i)}}x_{\eta(i)})={\sf s}_{u_{\eta(i)}}^{k_{\eta(i)}}
x_{\eta(i)}.$$
Finally, by ${\sf  c}_{u_{\eta(n-1)}}$ distributing over the boolean join, we get
$${\sf  c}_{u_{\eta(n-1)}} z_i=z_i \text { for }  i<n-1.$$

{\it Proof of ${\sf  c}_{u_{\eta(n-1)}}z_{n-1}=1.$}

Computing we get, by $u_{\eta(n-1)}\notin \Delta x_{\xi(n-1)}$
and by \cite[1.5.8(i), 1.5.8(ii)]{HMT1}
the following:
$${\sf  c}_{u_{\eta(n-1)}}(-{\sf  c}_{k_{\eta(n-1)}}x_{\eta(n-1)}
+ {\sf s}_{u_{\eta(n-1)}}^{k_{\eta(n-1)}}x_{\eta(n-1)})$$
$$ ={\sf  c}_{u_{\eta(n-1)}}-{\sf  c}_{k_{\eta(n-1)}}x_{\eta(n-1)}+ {\sf  c}_{u_{\eta(n-1)}}
{\sf s}_{u_{\eta(n-1)}}^{k_{\eta(n-1)}}x_{\eta(n-1)}$$

$$=-{\sf  c}_{k_{\eta(n-1)}}x_{\eta(n-1)}+ {\sf  c}_{u_{\eta(n-1)}}{\sf s}_{u_{\eta(n-1)}}^{k_{\eta(n-1)}}
x_{\eta(n-1)}$$
$$=-{\sf  c}_{k_{\eta(n-1)}}x_{\eta(n-1)}+ {\sf  c}_{u_{\eta(n-1)}}{\sf s}_{u_{\eta(n-1)}}^{k_{\eta(n-1)}}
{\sf  c}_{u_{\eta(n-1)}}x_{\eta(n-1)}$$
$$=-{\sf  c}_{k_{\eta(n-1)}}x_{\eta(n-1)}+ {\sf  c}_{k_{\eta(n-1)}}{\sf s}_{k_{\eta(n-1)}}^{u_{\eta(n-1)}}
{\sf  c}_{u_{\eta(n-1)}}x_{\eta(n-1)}$$
$$=-{\sf  c}_{k_{\eta(n-1)}}x_{\eta(n-1)}+ {\sf  c}_{k_{\eta(n-1)}}{\sf  c}_{u_{\eta(n-1)}}x_{\eta(n-1)}$$
$$= -{\sf  c}_{k_{\eta(n-1)}}x_{\eta(n-1)}+ {\sf  c}_{k_{\eta(n-1)}}x_{\eta(n-1)}=1.$$
With this the proof of (i) in our claim is complete.
Now we prove (ii).
Let $j<m$ . Then we have
$${\sf  c}_{u_{\eta(n-1)}}^{\partial}(-{\sf  c}_{l_{\xi(j)}}y_{\xi(j)})
=-{\sf  c}_{l_{\xi(j)}}y_{\xi(j)}$$ and
$${\sf  c}_{u_{\eta(n-1)}}^{\partial}
({\sf s}_{v_{\xi(j)}}^{l_{\xi(j)}}y_{\xi(j)})={\sf s}_{v_{\xi(j)}}^{l_{\xi(j)}}y_{\xi(j)}.$$
Indeed,  computing we get
$${\sf  c}_{u_{\eta(n-1)}}^{\partial}(-{\sf  c}_{l_{\xi(j)}}y_{\xi(j)})
=-{\sf  c}_{u_{\eta_{n-1}}}-(-{\sf  c}_{l_{\xi(j)}}y_{\xi(j)})
= -{\sf  c}_{u_{\eta(n-1)}}{\sf  c}_{l_{\xi(j)}}y_{\xi(j)}
=-{\sf  c}_{l_{\xi(j)}}y_{\xi(j)}.$$
Similarly,  we have
$${\sf  c}_{u_{\eta(n-1)}}^{\partial} ({\sf s}_{v_{\xi(j)}}^{l_{\xi(j)}}y_{\xi(j)})
=-{\sf  c}_{u_{\eta(n-1)}}- ({\sf s}_{v_{\xi(j)}}^{l_{\xi(j)}}y_{\xi(j)})$$
$$=-{\sf  c}_{u_{\eta(n-1)}} ({\sf s}_{v_{\xi(j)}}^{l_{\xi(j)}}-y_{\xi(j)})
=- {\sf s}_{v_{\xi(j)}}^{l_{\xi(j)}}-y_{\xi(j)}
= {\sf s}_{v_{\xi(j)}}^{l_{\xi(i)}}y_{\xi(j)}.$$
By ${\sf  c}_i^{\partial}({\sf  c}_i^{\partial}x+y)= {\sf  c}_i^{\partial}x+{\sf  c}_i^{\partial}y$
we get from the above that
$${\sf  c}_{u_{\eta(n-1)}}^{\partial}(t_j)={\sf  c}_{u_{\eta(n-1)}}^{\partial}({\sf  c}_{l_{\xi(j)}}y_{\xi(j)}+{\sf s}_{v_{\xi(j)}}^{l_{\xi(j)}}y_{\xi(j)})$$
$$={\sf  c}_{u_{\eta(n-1)}}^{\partial}{\sf  c}_{l_{\xi(j)}}y_{\xi(j)}+{\sf  c}_{u_{\eta(n-1)}}^{\partial}
{\sf s}_{v_{\xi(j)}}^{l_{\xi(j)}}y_{\xi(j)}$$
$$={\sf  c}_{l_{\xi(j)}}y_{\xi(j)}+ {\sf s}_{v_{\xi(j)}}^{l_{\xi(j)}}y_{\xi(j)}=t_j.$$

By the above proven claim we have
$${\sf  c}_{u_{\eta(n-1)}}\prod_{i<n}z_i={\sf  c}_{u_{\eta(n-1)}}[\prod_{i<n-1}z_i.z_n]$$
$$={\sf  c}_{u_{\eta(n-1)}}\prod_{i<n-1}z_i.
{\sf  c}_{u_{\eta(n-1)}}z_{n-1}=\prod_{i<n-1}z_i.$$
Here we are using that ${\sf  c}_i({\sf  c}_ix.y)={\sf  c}_ix.{\sf  c}_iy$.
Combined with (2) we obtain
$$a. \prod_{i<n-1}(-{\sf  c}_{k_{\eta(i)}}x_{\eta(i)}+{\sf s}_{u_{\eta(i)}}^{k_{\eta(i)}}x_{\eta(i)})
 \leq {\sf  c}_{u_{\eta(n-1)}}b_0.$$
On the other hand, from our claim and (3),
 it follows that
$$(-c).\prod_{j<m}
(-{\sf  c}_{l_{\xi(j)}}y_{\xi(j)}+{\sf s}_{v_{\xi(j)}}^{l_{\xi(j)}}y_{\xi(j)})\leq {\sf  c}_{u_{\eta(n-1)}}^{\partial}b_1.$$
Now making use of the induction hypothesis we get
$${\sf  c}_{u_{\eta(n-1)}}b_0.{\sf  c}_{u_{\eta(n-1)}}^{\partial}b_1\neq 0;$$
and hence that
$$b_0.{\sf  c}_{u_{\eta(n-1)}}^{\partial}b_1\neq 0.$$
From
$$b_0.{\sf  c}_{u_{\eta(n-1)}}^{\partial}b_1\leq b_0.b_1$$
we reach the desired conclusion, i.e. that
$$b_0.b_1\neq 0.$$
The other case, when $\eta(n-1)\leq \xi(m-1)$ can be treated analgously
and is therefore left to the reader.
We have proved that $H$ is a proper filter.

Proving that $H$ is a proper filter of $\Sg^{\cal B}(X_1\cap X_2)$,
let $H^*$ be a (proper boolean) ultrafilter of $\Sg^{\cal B}(X_1\cap X_2)$
containing $H.$
We obtain  ultrafilters $F_1$ and $F_2$ of $\Sg^{\cal B}(X_1)$ and $\Sg^{\cal B}(X_2)$,
respectively, such that
$$H^*\subseteq F_1,\ \  H^*\subseteq F_2$$
and (**)
$$F_1\cap \Sg^{\cal B}(X_1\cap X_2)= H^*= F_2\cap \Sg^{\cal B}(X_1\cap X_2).$$
Now for all $x\in \Sg^{\cal B}(X_1\cap X_2)$ we have
$$x\in F_1\text { if and only if } x\in F_2.$$
Also from how we defined our ultrafilters, $F_i$ for $i\in \{1,2\}$ satisfy the following
condition:

(*) For all $k<\mu$, for all $x\in Sg^{\cal B}X_i$
if ${\sf  c}_kx\in F_i$ then ${\sf s}_l^kx$ is in $F_i$ for some $l\notin \Delta x.$
We obtain  ultrafilters $F_1$ and $F_2$ of $\Sg^{\B}X_1$ and $\Sg^{\B}X_2$, 
respectively, such that 
$$H^*\subseteq F_1,\ \  H^*\subseteq F_2$$
and (**)
$$F_1\cap \Sg^{\B}(X_1\cap X_2)= H^*= F_2\cap \Sg^{\B}(X_1\cap X_2).$$
Now for all $x\in \Sg^{\cal B}(X_1\cap X_2)$ we have 
$$x\in F_1\text { if and only if } x\in F_2.$$ 
Also from how we defined our ultrafilters, $F_i$ for $i\in \{1,2\}$ are perfect.  

Then define the homomorphisms, one on each subalgebra, like in \cite{t} p. 128-129, using the perfect ultrafilters, 
then freeness will enable pase these homomophisms, to a single one defined to the set of free generators, 
which we can assume to be, without any loss, to 
be $X_1\cap X_2$ and it will satisfy  $h(a.-c)\neq 0$ which is a contradiction.

\end{proof}

\section{ Second proof}

A technique used in constructing cylindric algebras with certain desirable properties
is to construct  atom structure with certain first order corresponding properties 
(like Monk's algebras, Maddux's algebras, and the Hirsch-Hodkinson numerous 
constructions. Sayed Ahmed, too, has implemented such constructions in a few publications of his.). 

Algebraic logicians and Modal logicians, frequently talk about the same thing with different names 
which has caused a communication problem in the past.

This was partially surpassed  by the pioneering work of Venema on cylindric modal logic, later enhanced by the 
Van Benthem-Andr\'eka-N\'emeti-Goldblatt (and their students, to name a few, Marx, Mikulas, and Kurusz) 
colloboration. 

We know how to build atom structures,  or indeed frames, from atomic algebras, 
and conversely subalgebras of complex algebras from frames. 
We are happy when we are able to preserve crucial logical properties.

But something seems to be missing. Modal logicians rarely study frames in isolation, 
rather they are interested in constructing new frames from old ones using bounded morphisms, generated subframes, 
disjoint unions, zig-zag products (the latter is a less familiar notion). 

An algebraic logician adopts an analogous perspective but on 
different (algebraic) level, via such constructions as
homomorphisms, subalgebras and direct products. So it appears that modal logicians work in a universe 
that is distant from that of the algebraic logicians. In this connection, 
it is absolutely natural to ask whether these universes are perhaps systematically related. And the answer is: 
they are, very much so, and {\it duality theory} is devoted to studying such links.
For example a representation theorem for algebras is the dual of representing abstract state frames by what 
Venema calls assignment frames.

Here we give an application of this duality, worked out by Marx, to show that cylindric-polyadic algebras have the $SUPAP$.
This proof also works 
for all varieties of relativized cylindric polyadic algebras studied by Ferenzci and reported in \cite{Fer}. 

The result follows from the simple observation that such varieties can be axiomatized with positive, hence Sahlqvist equations, 
and therefore they are canonical; and also we do not have a Rosser condition on cylindrifiers; cylindrifiers do not commute, 
this allows that the first order correspondants of such equations are {\it clausifiable}, see \cite{Marx} for the 
definition of this. 

This proof is inspired by the modal perspective of cylindric-like algebras, applied to cylindric-polyadic algebras 
that suggests a whole landscape below standard predicate logic, 
with a minimal modal logic at the base ascending to standard semantics via frame constraints. In particular, 
this landscape contains nice sublogics of the full predicate logic, sharing its desirable meta properties and at the same time avoiding its 
negative accidents due to its Tarskian 'square frames' modelling. In such nice numerous sublogics cuantifiers do not commute, 
and the merry go round identities hold.
Such mutant logics are currently a very rich area of research.

The technique used here can be traced back to N\'emeti, when he proved that relativized cylindric set algebras have $SUPAP$; 
using (classical) duality
between atom structures and cylindric algebras. Marx 'modalized' the proof, and slightly strenghtened N\'emeti results, using instead 
the well-established duality between modal frames and 
complex algebras. 

While N\'emeti talks about subalgebras of {\it finite direct products} of frames, 
Marx talks about {\it finite zig-zag products} of frames, and this is a non-trivail very useful generalization for proving
strong amalgamation for a lot of relativized set algebras whose units satisfy 
certain closure properties. 

On the other hand, such strong amalgmation results cannot be obtained from 
N\'emeti's technique which seems to work only
for {\it very relativized} set algebras, namely, the class ${\bf Crs_{\alpha}}$, for any $\alpha$.
This class is referred to as the class of cylindric relativized set algebras in \cite{HMT2}.  
As a matter of fact, thislatter class is just the first step along a radical path, obtained by deconstructing, so to speak,
the semantics of first order
logic , designing lighter modal versions of this system by locating implicit choice points in
this step up. This gives a whole landscape of decidable version with different computational 
constraints over universes of states (assignments) related by variable updates. An off shoot of such a arelativization
technique is also
regaining definabity (interpolation) and for tha matter an easior match betwen syntax and semantics.

We consider the non-commutive cylindric polyadic algebras introduced by Ferenzci; we consider, this time, the  case $\CPEA$.
There is no deep motivation for such a choice, except that varying the studies algebras, suggests tht our proofs work for all.a 
A frame is a first order structure $\F=(F,  T_{i}, S_{\tau})_{i\in \alpha, \tau\in {}^{\alpha}\alpha}$ where $T$ is an arbitrary set and
and  both $T_{i}$ and $S_{\tau}$  are binary relations on $T$  for all $i\in \alpha$; and $\tau\in {}^{\alpha}\alpha$.

Given a frame $\F$, its complex algebra will be denotet by $\F^+$; $\F^+$ is the algebra $(\wp(\F), c_i, s_{\tau})_{i\in \alpha, \tau\in {}^{\alpha}\alpha}$  
where for $X\subseteq  V$,
$c_i(X)=\{s\in V: \exists t\in X, (t, s)\in T_i \}$, and similarly for $s_{\tau}$.

For $K\subseteq \CPEA_{\alpha}$, we let $\Str K=\{\F: \F^+\in K\}.$ 

For a variety $V$, it is always the case that 
$\Str V\subseteq \At V$ and equality holds if the variety is atom-canonical. 
If $V$ is canonical, then $\Str V$ generates $V$ in the strong sense, that is 
$V= {\bf S}\Cm \Str V$. For Sahlqvist varieties, that are completely additive, as is our case, $\Str V$ is elementary.

\begin{definition}
\item Given a family $(\F_i)_{i\in I}$ of frames, a {\it zigzag  product} of these frames is a substructure of $\prod_{i\in I}\F_i$ such that the
projection maps restricted to $S$ are
onto.
\end{definition}

\begin{definition} Let $\F, \G, \H$ be frames, and $f:\G\to \F$ and $h:\F\to \H$. 
Then $INSEP=\{(x,y)\in \G\times \H: f(x)=h(y)\}$. 
\end{definition}

\begin{lemma} The frame $INSEP \upharpoonright G\times H$ is a zigzag product
of $G$ and $H$, such that $\pi\circ \pi_0=h\circ \pi_1$, where $\pi_0$ and $\pi_1$ are the projection maps.
\end{lemma}
\begin{proof} \cite{Marx} 5.2.4
\end{proof}
For an algebra $\A$, $\A_+$ denotes its ultrafilter atom structure. 
For $h:\A\to \B$, $h_+$ denotes the function from $\B_+\to \A_+$ defined by $h_+(u)=h^{-1}[u]$ 
where the latter is $\{x\in a: h(x)\in u\}.$

\begin{theorem}(\cite{Marx} lemma 5.2.6)
Assume that $K$ is a canonical variety and $\Str K$ is closed under finite zigzag products. Then $K$ has the superamalgamation
property.
\end{theorem}
\begin{demo}{Sketch of proof} Let $\A, \B, \C\in K$ and $f:\A\to \B$ and $h:\A\to \C$ be given monomorphisms. 
Then $f_+:\B_+\to \A_+$ and $h_+:\C_+\to \A_+$. We have $INSEP=\{(x,y): f_+(x)=h_+(y)\}$ is a zigzag connection. Let $\F$ be the zigzag product 
of $INSEP\upharpoonright \A_+\times \B_+$. 
Then $\F^+$ is a superamalgam.
\end{demo}

\begin{theorem}
The variety $\CPEA_{\alpha}$ has $SUPAP$.
\end{theorem}
\begin{proof} $\CPEA_{\alpha}$ can be easily  defined by positive equations then it is canonical.
The first order correspondents of the positive equations translated to the class of frames will be Horn formulas, hence clausifiable \cite{Marx} theorem 
5.3.5, and so $\Str K$ is closed under finite zigzag products. 
Marx's theorem finishes the proof.
\end{proof}

If one views relativized models as the natural semantics for predicate logic rather than some tinkering 
devise which is the approach adopted in \cite{HMT1}, then 
many well -established taboos of the field must be challenged. 

In standard textbooks one learns that predicate logical validity is one unique notion specified once and for all by the usual 
Tarskian (square) semantics and canonized by G\"odel's completeness theorem. Moreover, it is essentially complex, 
being undecidable by Church's theorem.

On the present view, however standard predicate logic has arisen historically by making several ad-hoc semantic decisions 
that could have gone differently. Its not all about 'one completeness theorem' but rather about several completeness theorems
obtained by varying both the semantic and syntactical parameters. 
This can be implemented from a classical relativized representability theory, 
like that adopted in the monograph \cite{HMT1}, though such algebras were treated in op.cit off main stream, 
and they were only brought back to the front of the scence by the work of Resek, Thompson, 
Andr\'eka and last but not least Ferenczi,  or from a modal perspective, that has enriched the subject 
considerably.

But on the other  hand, careful scrutiny of the situation reveals that things are not so clear cut, and the borderlines are hazy. 
Within the polyadic cylindric dichotomy there is the square relativisation dichotomy, and also 
vice versa.





\begin{thebibliography}{}

\bibitem{1} H. Andreka, M. Ferenczi, I. Nemeti (editors) {\it Cylindric-like algebras and algebraic logic} Bolyai Society, Mathematical Studies, Springer
(2013).

\bibitem{DM} Daigneault, A., and Monk,J.D., 
{\it Representation Theory for Polyadic algebras}. 
Fund. Math. {\bf 52}(1963) p.151-176.

\bibitem{Fer} Ferenzci {\it Cylindric polyadic algebras} Pre-print

\bibitem{Fer1} M. Ferenczi {\it A new representation theory, representing cylindric like algebras by relativized set algebras} In \cite{1}p. 135-162

\bibitem{HMT1} L. Henkin, D. Monk, A. Tarski {\it Cylindric algebras, part 1} 1970

\bibitem{HMT2} L. Henkin, D. Monk, A. Tarski {\it Cylindric algebras, part 2} 1985

\bibitem{Marx} Marx {\it Algebraic relativization and arrow logic} Ph d Dissertation. University of 
Amsterdam 1995.

\bibitem{Sagi} G. Sagi {\it Polyadic algebras} In \cite{1}p. 376-392




\bibitem{t}T. Sayed Ahmed {\it Neat reducts and Neat Embeddings in Cylindric Algebras} in \cite{1},  p. 105-134




\end{thebibliography}
\end{document}